\documentclass{amsart}

\usepackage{amssymb,amsmath,amsthm,latexsym,amscd,amsfonts,mathrsfs,mathtools}

\usepackage[margin=3.7cm]{geometry}
\usepackage{ulem}
\usepackage{epsfig}
\usepackage{psfrag}
\usepackage{graphicx}
\usepackage{enumerate}

\usepackage{color}
\usepackage{hyperref}
\usepackage{cleveref}

\newtheorem{theorem}{Theorem}[section]

\newtheorem{proposition}[theorem]{Proposition}
\newtheorem{corollary}[theorem]{Corollary}
\newtheorem{lemma}[theorem]{Lemma}

\theoremstyle{definition}
\newtheorem{definition}[theorem]{Definition}

\newtheorem{example}[theorem]{Example}
\newtheorem{remark*}[theorem]{}

\theoremstyle{remark}
\newtheorem{remark}[theorem]{Remark}

\newcommand{\C}{\mathbb{C}}
\newcommand{\Z}{\mathbb Z}
\newcommand{\N}{\mathbb N}

\newcommand{\mM}{\mathcal{M}}
\newcommand{\mN}{\mathcal{N}}
\newcommand{\mQ}{\mathcal{Q}}

\newcommand{\mZ}{\mathcal{Z}}

\newcommand{\mU}{\mathcal{U}}
\newcommand{\mL}{\mathcal{L}}
\newcommand{\mA}{\mathcal{A}}
\newcommand{\mB}{\mathcal{B}}
\newcommand{\mC}{\mathcal{C}}
\newcommand{\mH}{\mathcal{H}}

\newcommand{\tr}{\mathrm{tr}}

\newcommand{\Ind}{\mathrm{Ind}}

\begin{document}
 \title{Regular inclusions of simple unital $C^*$-algebras}

 \author[K C
   Bakshi]{Keshab Chandra Bakshi} \address{Dept. of Mathematics, Indian
   Institute of Technology Kanpur, Kanpur, UP, INDIA}
 \email{bakshi209@gmail.com, keshab@iitk.ac.in}
 
\thanks{The first named author was supported through a DST
  INSPIRE faculty grant (reference  no. DST/INSPIRE/04/2019/002754).}

\author[V P Gupta]{Ved Prakash Gupta} \address{School of Physical
  Sciences, Jawaharlal Nehru University, New Delhi, INDIA}
\email{vedgupta@mail.jnu.ac.in, ved.math@gmail.com} \maketitle

\begin{abstract}
  We prove that an inclusion $\mB \subset \mA$ of simple unital
  $C^*$-algebras with a finite-index conditional expectation is
  regular if and only if there exists a finite group $G$ that admits a
  cocycle action $(\alpha,\sigma)$ on the intermediate
  $C^*$-subalgebra $\mC$ generated by $\mB$ and its centralizer
  $\mC_\mA(\mB)$ such that $\mB$ is outerly $\alpha$-invariant
  and $(\mB \subset \mA) \cong ( \mB \subset
  \mC\rtimes^r_{\alpha, \sigma} G)$.

Prior to this characterization, we prove the existence of
two-sided and unitary quasi-bases for the minimal conditional
expectation of any such  inclusion, and also show that such an inclusion
has  integer Watatani index  and depth at most $2$.
  \end{abstract}

\section{Introduction}

Of late, the study of inclusions of $C^*$-algebras has attracted a
good deal of attention - see, for instance, \cite{Wat, Iz, Iz2, IW,
  JOPT, Ren, R, BG2} and the references therein. Interestingly, people
have focused on different perspectives of such inclusions and the
project has thrived in the recent years. In fact, certain
fundamental structure results have been achieved during the last 25
years or so. Among those results, Kajiwara-Watatani's (Goldman type)
characterization of index-2 inclusions of $C^*$-algebras as
fixed-point inclusions via outer actions of the cyclic group $\Z_2$
(\cite[Theorem 5.13]{KW}), Renault's characterization of a Cartan pair
$\mB \subset \mA$ of $C^*$-algebras (\cite[Theorem 5.9]{Ren}) and
Izumi's characterization of depth-2 inclusions of simple
$C^*$-algebras as fixed-point inclusions via outer actions by finite
dimensional Kac algebras (\cite[Corollary 6.4]{Iz}) are noteworthy and
serve as precursors to the theme of this article.

Our approach in \cite{BG2} as well
as here is mainly motivated by the Fourier theoretic aspects and
certain structure results of the theory of subfactors, a theory which
was initiated by Vaughan Jones in the decade of 80s - see \cite{Jones,
  JS} for a quick introduction. More precisely, in \cite{BG2},
depending heavily on the fundamental work \cite{Wat} of Watatani, we
made an attempt to develop a Fourier theory on the tower of relative
commutants (which was augmented further in \cite{BGS}) and, motivated
by \cite{BDLR}, we also introduced a notion of angle between
compatible intermediate $C^*$-subalgebras of an inclusion of unital
$C^*$-algebras with a finite-index conditional expectation. In this
article, we present a structure result for {\it regular} inclusions of
simple unital $C^*$-algebras.

Recall that an inclusion $\mQ \subset \mM$ of von Neumann algebras is
said to be regular if the normalizer of $\mQ$ in $\mM$,
$\mN_\mM(\mQ):=\{u \in \mU(\mM) : u \mQ u^* = \mQ\}$, generates $\mM$
as a von Neumann-algebra.  For any irreducible regular inclusion
$N\subset M$ of factors, it is well known that its Weyl group
$G:=\frac{\mN_M(N)}{\mU(N)}$ admits an outer cocycle action $(\alpha,
\sigma)$ on $N$ and $(N \subset M) \cong (N \subset N
\rtimes_{(\alpha, \sigma)} G)$. Moreover, if $N$ and $M$ are
$II_1$-factors, then it is known, thanks to the work of Sutherland
(\cite{Sut}), that this cocycle action can be untwisted, i.e., for any
regular irreducible inclusion $N \subset M$ of factors of type $II_1$
with finite (Jones) index, its Weyl group $G$ admits an outer action
on $N$ and $(N \subset M) \cong (N \subset N \rtimes G)$ (see
\cite{Kos, PiPo2, Hong, JoPo}). \smallskip

Analogous to the notion of regularity of an inclusion of von Neumann
algebras, we say that an inclusion $\mB \subset \mA$ of unital
$C^*$-algebras is {\it regular} if the normalizer of $\mB$ in $\mA$
generates $\mA$ as a $C^*$-algebra. The crux of this article lies in
proving the following, somewhat general, characterization of regular
inclusions of simple unital $C^*$-algebras:\smallskip

    \noindent {\bf \Cref{regular-characterization}. } {\it Let $\mB \subset
  \mA$ be an inclusion of simple unital $C^*$-algebras with a
  finite-index conditional expectation from $\mA$ onto $\mB$. Then, the
  inclusion $\mB \subset \mA$ is regular if and only if there exists a
  finite group $G$ that admits a cocycle action $(\alpha,\sigma)$ on
  the intermediate $C^*$-subalgebra $\mC:=C^*(\mB \cup \mC_\mA(\mB))$ such
  that \begin{enumerate}
  \item $\mB$ is invariant under $\alpha$;
  \item for each $e \neq g \in G$, $\alpha_g$ is an outer automorphism
    of $\mB$; and,
    \item $(\mB \subset \mA) \cong
      ( \mB \subset \mC\rtimes_{\alpha, \sigma}^r G)$.
\end{enumerate}}\smallskip

As a consequence, for any irreducible regular inclusion $\mB \subset
\mA$ of simple unital $C^*$-algebras with a finite-index conditional
expectation from $\mA$ onto $\mB$, its Weyl group
$G:=\frac{\mN_\mA(\mB)}{\mU(\mB)}$ admits an outer cocycle action
$(\alpha, \sigma)$ on $\mB$ such that $(\mB \subset \mA) \cong ( \mB
\subset \mB \rtimes^r_{(\alpha, \sigma)}G)$ - see
\Cref{characterization-irreducible-regular}.

The proof of \Cref{regular-characterization} relies on first showing
that the intermediate $C^*$-subalgebra $\mC$ generated by $\mB$ and its
centralizer $\mC_\mA(\mB):=\{x \in \mA: xb = bx \text{ for all } b \in \mB\}$
is compatible in the sense of Ino-Watatani (\cite{IW}), i.e., there
exists a (finite-index) conditional expectation $F: \mA \to \mC$ such that
${(E_0)}_{\restriction_\mC} \circ F = E_0$, where $E_0: \mA \to \mB$ is the
unique minimal conditional expectation as in \Cref{minimal} (see
\Cref{IMS-F}); and then, employing the conditional expectation $F$, we
obtain a natural outer cocycle action of the generalized Weyl group
$W(\mB \subset \mA) := \frac{\mN_\mA(\mB)}{\mU(\mB)\mU(\mC_\mA(\mB))}$  on the
$C^*$-subalgebra $\mC$, which leads to the desired characterization.
  
  The skeleton of this article emerged essentially from our ongoing
  attempts at obtaining a better understanding of regular inclusions
  of type $II_1$-factors (see \cite{BG1, BG3}).  While analyzing
  regular inclusions of $II_1$-factors, it was also established in
  \cite{BG3, CKP} that every regular inclusion of type $II_1$-factors
  $N \subset M$ with finite Jones index is always of depth at most
  $2$. Drawing motivation from this, prior to the characterization
  obtained in \Cref{regular-characterization}, we also show that a
  regular inclusion $\mB \subset \mA$ of simple unital $C^*$-algebras
  with a finite-index conditional expectation from $\mA$ onto $\mB$ is
  always of depth\footnote{Analogous to the notion of depth for
  inclusions of $II_1$-factors, the notion of depth for inclusions of
  unital $C^*$-algebras was formalized in \cite{JOPT}.} at most $2$
  (\Cref{regular-depth-2}).
    
    \Cref{regular-depth-2} is achieved on the lines of the proof of
    \cite[Theorem 4.3]{BG3} (also see \cite[Corollary 3.2]{CKP}),
    whose main ingredients are the following two observations, which
    are also of independent interest and  fundamental in
    nature:\smallskip

\noindent {\it If $\mB \subset \mA$ is a regular inclusion of simple unital
  $C^*$-algebras with a finite-index conditional expectation and $E_0:
  \mA \to \mB$ is the minimal conditional expectation as in
  \Cref{minimal}, then
  \begin{enumerate}
    \item $E_0$ admits a two-sided orthonormal quasi-basis in $\mA$
      (\Cref{weyl-group}); and,
\item $E_0$ admits a unitary orthonormal quasi-basis in $\mA$
  (\Cref{E0-restriction-to-C-unitary}).
  \end{enumerate}}

The proof of \Cref{weyl-group} borrows heavily from the methodology
employed in \cite{BG1}, wherein, among other results, using the notion
of path algebras by Sunder (and, independently, by Ocneanu), we had
provided a partial answer to a question of Jones by showing that if $N
\subset M$ is a regular inclusion of type $II_1$-factors with finite
Jones index, then the (unique) trace preserving conditional $E: M \to
N$ admits a two-sided Pimsner-Popa basis in $M$.

The proof of the fact that every regular finite-index inclusion of
$II_1$-factors is of depth at most 2 (demonstrated in \cite{BG3, CKP})
depended on the existence of a unitary orthonormal basis for the
unique trace preserving conditional expectation, which was essentially
achieved by reducing the problem to a finite-dimensional setup and
then employing some nice unitary matrices from the world of Quantum
Information Theory. Here as well, a suitable adaptation of those
techniques provides a proof of \Cref{E0-restriction-to-C-unitary},
which eventually allows us to deduce that every regular inclusion $\mB
\subset \mA$ of simple unital $C^*$-algebras with a finite-index
conditional expectation from $\mA$ onto $\mB$ has depth at most $2$
(\Cref{regular-depth-2}).

 Along the way, analogous to \cite[Theorem 3.12]{BG1}, we also prove that the
Watatani index of a regular inclusion $\mB \subset \mA$ of simple unital $C^*$-algebras
is the integer   $|W(\mB \subset \mA)|\dim (\mC_\mB(\mA)) $ - see \Cref{theorem-regular}.\color{black}
\section{Preliminaries}

\subsection{Watatani index and $C^*$-basic construction}
        \subsubsection{Watatani index}
 Consider an inclusion $\mB \subseteq \mA$ of unital $C^*$-algebras
 with a conditional expectation $E:\mA \rightarrow \mB$. A finite set
 $\{\lambda_i : 1 \leq i \leq n\} \subset \mA$ is said to be a {\it
   right} (resp., {\it left}) {\it quasi-basis} for $E$ if
 \[
 x=\sum_{i=1}^n
 E(x\lambda_i)\lambda^*_i\ (\text{resp.}, x = \sum_{i=1}^n E(x
 \lambda_i^*)\lambda_i)
 \]
 for every $x\in \mA$.  A finite collection $\{\lambda_i\}\subset \mA$ is
 said to be a {\it two-sided quasi-basis} for $E$ if it is both a left and a
 right quasi-basis for $E$. For example, if a right (or a left)
 quasi-basis for $E$ is a self-adjoint set, i.e., $\{\lambda_i^* \} =
 \{\lambda_i\}$, then it is a two-sided quasi-basis for $E$. Note that a finite set $\{\lambda_i: 1 \leq i \leq n\}$ is a right
quasi-basis for $E$ if and only if $\{\lambda_i^*: 1 \leq i \leq n\}$
is a left quasi-basis for $E$.

A conditional expectation $E: \mA \to
\mB$ is said to have \textit{finite index} (in the sense of Watatani)
if it admits a right (equivalently, a left) quasi-basis $\{\lambda_i:
1 \leq i \leq n\}$ in $\mA$.  Then, the right Watatani index of $E$ is
defined as
 \[
 \mathrm{Ind}_W^r (E)= \sum_{i=1}^n \lambda_i\lambda^*_i, 
 \]
 and is independent of the right quasi-basis.  Likewise, the left
 Watatani index of $E$ is defined as \( \mathrm{Ind}_W^l (E)=
 \sum_{i=1}^n \lambda_i^*\lambda_i\), where $\{\lambda_i\}$ is a left
 quasi-basis. Clearly,
 \[
 \mathrm{Ind}_W^r (E)= \mathrm{Ind}_W^l (E)
 \]
 and this common quantity is called the Watatani index of $E$ and is
 denoted by $\mathrm{Ind}_W(E)$.

In general, $\mathrm{Ind}_W (E)$ is not a scalar but it is an
invertible positive element of $\mZ(\mA)$.

Further, a right (resp., a left) quasi-basis $\{\lambda_i: 1 \leq i
\leq n\}$ of $E$ is said to be {\it orthonormal} if $E(\lambda_i^*
\lambda_i) = \delta_{i,j}$ (resp., $E(\lambda_i \lambda_j^*) =
\delta_{i,j}$) for all $1 \leq i, j \leq n$. And, a right (or a left)
quasi-basis of $E$ is said to be {\it unitary} if it consists of
unitary elements.

{\bf Note:} Following Watatani (\cite{Wat}), throughout this article, by a
 quasi-basis we shall simply mean a right quasi-basis, and not a
 two-sided quasi-basis!

\begin{remark}\label{fi-ce-facts}
                \begin{enumerate}
\item   A finite-index conditional expectation $E: \mA \to \mB$ is faithful
  and $1_\mB=E(1_\mA) = 1_\mA$.
                \item  If $\mB \subset \mC
                  \subset \mA$ is an inclusion of unital $C^*$-algebras
                  and, $E: \mA \to \mC$ and $F: \mC\to \mB$ are finite-index conditional
                  expectations with right (resp., left) quasi-bases
                  $\{\lambda_i\}$ and $\{\mu_j\}$, respectively, then
                  $\{\lambda_i \mu_j\}_{i,j}$ (resp., $\{
                  \mu_j\lambda_i\}_{j, i }$) is a right (resp., a
                  left) quasi-basis for $F \circ E$. (\cite[Proposition 1.7.1]{Wat})

                  In particular, the composition of two finite-index
                  conditional expectations $\mA \stackrel{E}{\to} \mC$ and $\mC
                  \stackrel{F}{\to} \mB$ is also a finite-index
                  conditional expectation and, moreover,  if $\mathrm{Ind}_W(F) \in \mZ(\mA)$, then it follows readily that
                  \[
                  \mathrm{Ind}_W(F \circ E) = \mathrm{Ind}_W(F)
                  \mathrm{Ind}_W(E).
                  \]
        \end{enumerate}
        \end{remark}

\subsubsection{Watatani's $C^*$-basic construction}

We now briefly recall the theory of \textit{$C^*$-basic construction}
introduced by Watatani in \cite{Wat}.

Consider an inclusion $\mB \subset \mA$ of unital $C^*$-algebras  with
common unit and a faithful conditional expectation $E: \mA \to \mB$.
Then, $\mA$ becomes a pre-Hilbert $\mB$-module with respect to the
$\mB$-valued inner product given by
  \begin{equation}\label{B-valued}
  \langle x, y\rangle_{\mB}=E(x^*y), x, y\in \mA.
  \end{equation}
Let $\mathfrak{A}$ denote the Hilbert $\mB$-module completion of $\mA$
and $\iota:\mA\rightarrow \mathfrak{A}$ denote the isometric inclusion
map with respect to the norm $\|x\|_\mA:=\|E(x^*x)^{1/2}\|$, $x \in
\mA$. The space $\mathcal{L}_{\mB}(\mathfrak{A})$ consisting of
adjointable $\mB$-linear maps on $\mathfrak{A}$ is a unital
$C^*$-algebra.

  For each $a \in \mA$, consider $\lambda(a)\in
  \mathcal{L}_\mB(\mathfrak{A})$ given by
  $\lambda(a)\big(\iota(x)\big)=\iota(ax)$ for $x\in \mA$. The map
  $\iota (\mA) \ni \iota(x) \mapsto \iota (E(x))\in \iota (\mA)$
  extends to an adjointable projection on $\mathfrak{A}$, and is
  denoted by $e_1\in \mathcal{L}_\mB(\mathfrak{A})$. The projection
  $e_1$ is called the Jones projection for the inclusion $\mB \subset
  \mA$; thus, $e_1(\iota(x))=\iota(E(x))$ for all $x \in
  \mA$. Moreover, $e_1 \in \lambda(\mB)'$ and it satisfies
  the relation $e_1\lambda(x) e_1 = \lambda(E(x)) e_1$ for all $x \in
  \mA$.

  The Watatani's $C^*$-basic construction for the inclusion $\mB
  \subset \mA$ with respect to the conditional expectation $E$ is
  defined as the $C^*$-subalgebra
  \[
\mA_1=  \overline{\mathrm{span}} \{\lambda(x) e_1
  \lambda(y) :x,y\in \mA\} \text{ of } \mathcal{L}_{\mB}(\mathfrak{A}).
  \]
  Also, $\lambda$ is an injective $*$-homomorphism and thus we can
  consider $\mA$ as a $C^*$-subalgebra of $\mA_1$. \color{black}

\begin{remark}\cite{Wat} \label{dual-ce}
With running notations and the indentification mentioned in the
preceding paragraph, if $E$ is a finite-index conditional expectation,
then the following hold:
   \begin{enumerate}
\item $\mA$ is complete with respect to the norm $\|\cdot\|_\mA$ - see
  \cite[Lemma 2.11]{BG2}.
\item  $\mA_1= \mathrm{span}\{x e_1 y : x, y \in \mA\} = C^*(\mA, e_1)$. (\cite[Propn. 1.3.3]{Wat}) \color{black}
\item 
  There exists a unique finite-index conditional expectation
  $\widetilde{E}: \mA_1 \to \mA$ satisfying
  \[
  \widetilde{E}\big( x e_1  y \big) =
   x \big(\mathrm{Ind}_W(E)^{-1}\big) y =
  \big(\mathrm{Ind}_W(E)^{-1} x y \big) 
  \]
 for all $x, y \in \mA$. ($\widetilde{E}$ is called the dual conditional expectation of
 $E$.)  (\cite[Prop. 1.6.1]{Wat})
  \end{enumerate}
  \end{remark}

        The following characterization of a (right) quasi-basis has
        proved to be extremely handy in the theory of subfactors and
        we have slowly started to realize that it is quite useful in
        the $C^*$-context as well. Its proof is most likely a
        folklore. For instance, its necessity is well known - see
        \cite[Lemma 2.1.6]{Wat} (also see \cite[Proposition
          2.12]{BG2}) - and its sufficiency follows on the lines of the
        proof of ($(2) \Rightarrow (3)$ of) \cite[Theorem 2.2]{Bak} by
        applying \cite[Lemma 3.7]{JOPT} (also see \cite[Lemma
          2.15]{BG2}), which is the $C^*$-analogue of the so-called
        ``Pushdown Lemma'' by Pimsner and Popa.

        \begin{proposition}\label{basis-characterization}
Let $\mB \subset \mA$ be an inclusion of unital $C^*$-algebras, $E:
\mA \to \mB$ be a finite-index conditional expectation and $\mA_1$ be
the Watatani's $C^*$-basic construction of $\mB \subset \mA$ with
respect to $E$ and, $e_1$ be the corresponding Jones projection. Then,
a finite set $\{\lambda_i : 1 \leq i \leq n\}\subset \mA$ is a (right)
quasi-basis for $E$ if and only if $\sum_{i=1}^n \lambda_i e_1
\lambda_i^* = 1$.
          \end{proposition}

The following result by Watatani is fundamental in nature and was used
extensively in developing the Fourier theory for inclusions of simple
unital $C^*$-algebras in \cite{BG2, BGS}.
        \begin{theorem}\cite[Corollary 2.2.14]{Wat}\label{simple-A1}
            If $\mB$ is a simple $C^*$-subalgebra of a  unital
   $C^*$-algebra $\mA$ and $E: \mA \to \mB$ is a finite-index
  conditional expectation, then the $C^*$-basic construction
  $\mA_1$ of $\mB \subset \mA$ with respect to $E$ is also simple.
\end{theorem}
\subsection{Markov trace and two-sided quasi-basis}

Recall that, for a unital inclusion $Q \subset P$ of finite dimensional
$C^*$-algebras with inclusion matrix $\Lambda$, a tracial state $\tau$
on $P$ with (column) trace vector $\bar{t}$ is said to be a Markov trace for
the inclusion $Q \subset P$  with modulus $\beta>0$ if $\Lambda^t
\Lambda \bar{t} = \beta \bar{t}$. If $Q \subset P$ is a connected
unital inclusion, then there exists a unique Markov trace for the
inclusion $Q \subset P$ with modulus $\|\Lambda\|^2$. For more on
Markov traces, see \cite{JS, GHJ}.

\begin{remark}\label{Markov-trace-vector}
  If $P$ is a finite dimensional $C^*$-algebra with dimension vector
  $[n_1 , \ldots, n_k]$, then the inclusion matrix of any unital
  inclusion $\C \subset P$ is given by $\Lambda = [n_1, \ldots,
    n_k]\in M_{1, k}(\N)$ and the unique Markov trace for this
  inclusion with modulus $\|\Lambda \|^2$ has trace vector
  $\bar{t}=(\frac{n_1}{d}, \ldots, \frac{n_k}{d})^t$, where $d :=
  \|\Lambda\|^2 = \mathrm{dim}(P)$.
  \end{remark}

The first part of the following observation was essentially made in
\cite{BG1} and was derived by employing the notion of path algebras
associated to an inclusion of finite-dimensional $C^*$-algebras
(introduced independently by Ocneanu and Sunder).  We rephrase and
reproduce some portion of it to suit the requirements of this article.
(We must mention that the approach of path algebras was employed by
Watatani as well to provide an example of a (self-adjoint and hence
two-sided) quasi-basis for the trace-preserving conditional
expectation from a finite dimensional $C^*$-algebra onto a subalgebra
- see \cite[Lemma 2.4]{Wat}.)
    \begin{proposition}\cite{BG1}\label{tr-2-sided-basis}
      Let $P$ be a finite dimensional $C^*$-algebra and $\tr$ be a
      faithful tracial state on $P$. Then, there exists a two-sided
      orthonormal quasi-basis for $\tr$ and
         \[
      \mathrm{Ind}_W(\tr) =\sum_{i=1}^k \frac{n_i^2}{\tr(p_i)} p_i, 
      \]
      where $\{p_i, 1\leq i \leq k\}$ is the set of minimal central
      projections of $P$ with $n_i^2 =\mathrm{dim}(p_i\mA)$ for all $1
      \leq i \leq k$.

Moreover, $\mathrm{Ind}_W(\tr)$ is a scalar if and only if
  $\tr$ is the Markov trace for the inclusion $\C \subset P$ with
  modulus $\mathrm{dim}(P)$. And, in that case, $\mathrm{Ind}_W(\tr) =
  \mathrm{dim}(P)$.
    \end{proposition}
    \begin{proof}
    
    That $\tr$ admits a two-sided orthonormal quasi-basis follows
    verbatim on the lines of the proof of \cite[Proposition 3.3]{BG1}
    (see also \cite[Lemma 3.11]{BG3}).

 Moreover, from the proof of \cite[Proposition 3.3]{BG1}, it also
    follows that there exists a system of matrix units
    $\{e^{(i)}_{(\kappa, \beta)}:1 \leq \kappa, \beta \leq n_i, 1 \leq
    i \leq k\}$ for $P$ such that
    \[
    \Big\{
    \sqrt{\frac{n_i}{\tr(p_i)}}e^{(i)}_{(\kappa, \beta)}: 1
    \leq \kappa, \beta \leq n_i, 1\leq i\leq k\Big\}
    \]
    is a two-sided orthonormal quasi-basis for $\tr$. Hence,
         \[
         \mathrm{Ind}_W(\tr) = \sum_{i = 1}^k
         \sum_{\{1 \leq \kappa, \beta \leq n_i\}}
         \frac{\sqrt{n_i}}{\sqrt{\tr(p_i)}}e^{(i)}_{(\kappa,
           \beta)}\frac{\sqrt{n_i}}{\sqrt{\tr(p_i)}}(e^{(i)}_{(\kappa,
           \beta)})^* =\sum_{i=1}^k \frac{n_i^2}{\tr(p_i)} p_i .
         \]

Next, suppose that $\tr$ is the Markov trace for the inclusion $\C
\subset P$ with modulus $d = \mathrm{dim}(P)$, then its trace vector
is given by $\bar{t} = (\frac{n_1}{d}, \ldots,
\frac{n_k}{d})^t$. Thus, $\tr(p_i) = \frac{n_i^2}{d}$ for all
$i$. Hence,
\[
\mathrm{Ind}_W(\tr) = \sum_{i=1}^k \frac{n_i^2}{\tr(p_i)} p_i =
\sum_{i=1}^k d p_i = d .
\]

Conversely, suppose that $\mathrm{Ind}_W(\tr)$ is a
scalar, say, $c>0$. Then $\sum_{i=1}^k \frac{n_i^2}{\tr(p_i)} p_i = c
= \sum_{i=1}^k c p_i$, which implies that $\frac{n_i^2}{\tr(p_i)} = c$
for all $i$; so that, $\tr(p_i) = \frac{n_i^2}{c}$ for all $i$. So,
$\bar{s}:=(\frac{n_1}{c}, \ldots, \frac{n_k}{c})^t$ is the trace
vector of $\tr$. Also, the inclusion matrix for $\C \subset P$ is
given by $\Lambda = [n_1, \ldots, n_k]$ so that $ \|\Lambda\|^2 =
\sum_{i=1}^k n_i^2 = \mathrm{dim}(P)$. Further,   we observe that $\Lambda^t
\Lambda\bar{s} = \|\Lambda\|^2\, \bar{s}$, i.e., $\tr$ is the Markov
trace for the inclusion $\C \subset P$ with modulus
$\mathrm{dim}(P)$. (Further, by its uniqueness, it follows that
$\bar{s} = \bar{t}$ and hence that $c = d$.) 
    \end{proof}
 The preceding observation is relatable to \cite[Proposition 2.4.2 and
   Corollary 2.4.3]{Wat} as well.

    \subsection{ Minimal conditional expectation}
 Note that, for an inclusion $\mB \subset \mA$ of unital
 $C^*$-algebras, if $\mZ(\mA) = \C$, then every finite-index
 conditional expectation from $\mA$ onto $\mB$ has a scalar Watatani
 index. Recall from \cite{Wat} that a conditional expectation is said
 to be minimal if it has the smallest Watatani index.  In general, a
 minimal conditional expectation from $\mA$ onto $\mB$ need not
 exist. However,  when it comes to inclusions of
 simple unital $C^*$-algebras, a minimal conditional expectation
 exists and is, in fact, unique.

        \begin{theorem}\cite[Theorem 2.12.3]{Wat}\label{minimal}
Let $\mB \subset \mA$ be an inclusion of simple unital $C^*$-algebras with
a finite-index conditional expectation from $\mA$ onto $\mB$. Then, there
exists a unique minimal conditional expectation from $ \mA$ onto $ \mB$
(which is usually denoted by $E_0$).
  \end{theorem}

\begin{definition} \cite{Wat}
  Let $\mB \subset \mA$ be as in the preceding theorem. Then, the Watatani index of the inclusion $\mB \subset \mA$ is defined as \[
  [\mA:\mB]_0 = \mathrm{Ind}_W(E_0).
  \]
\end{definition}

 \begin{remark}\label{dual-index}
   For $\mA$, $\mB$ and $E_0$ as in \Cref{minimal}, \\ (a) the
   $C^*$-basic construction $\mA_1$ is also simple - see \cite[Corollary
     2.2.14]{Wat});\\ (b) the dual conditional expectation
   $\widetilde{E}_0: \mA_1 \to \mA$ is also minimal  - see \cite[Corollary
     3.4]{KW}); and,\\ (c) $\mathrm{Ind}_W
   (E_0) = \mathrm{Ind}_W (\widetilde{E}_0)$  - see \cite[Proposition
     2.3.4]{Wat}.
 \end{remark}

 For any inclusion $\mB \subset \mA$ of algebras, recall that the centralizer of
 $\mB$ in $\mA$ is given by \[
 \mC_\mA(\mB) = \{ x \in \mA: xb = bx
 \text{ for all } b \in \mB\},
 \]
 which is also called the relative
 commutant of $\mB$ in $\mA$ and denoted by $\mB'\cap \mA$.
 
\begin{proposition}\label{E0-restriction-to-C-two-sided} 
Let $\mB \subset \mA$ be an inclusion of simple unital $C^*$-algebras with a
finite-index conditional expectation from $\mA$ onto $\mB$ and let $E_0: \mA
\to \mB$ denote the unique minimal conditional expectation as in
\Cref{minimal}. Then, the following hold:
\begin{enumerate}
    \item \cite{Wat} For any quasi-basis $\{\lambda_i : 1 \leq i \leq
      n\}$ of $E_0$,
    \[
    E_0(x) =  \frac{1}{[\mA:\mB]_0}\sum_{i=1}^n \lambda_i x \lambda_i^* \text{ for all
  } x \in \mC_\mA(\mB),
  \]
   $E_0(\mC_\mA(\mB)) = \C$ and $\tau_0:= {E_0}_{\restriction_{\mC_\mA(\mB)}}$
  is a faithful tracial state on $\mC_\mA(\mB)$. 
\item Let $\mC:= C^*(\mB \cup \mC_\mA(\mB) )$. Then, the conditional expectation
  ${E_0}_{\restriction_C}:C\to \mB$ admits a two-sided orthonormal
  quasi-basis contained in $\mC_\mA(\mB)$.
  \end{enumerate}
    \end{proposition}

\begin{proof}

(1):  Let $\{\lambda_i: 1 \leq i \leq n\}$ be a quasi-basis for
  $E_0$. Since $E_0$ is minimal, it follows from \cite[Proposition
    1.2.9 and Theorem 2.12.3]{Wat} that
    \begin{enumerate}[(a)] 
    \item $E_0(x) =  \frac{1}{[\mA:\mB]_0}\sum_{i=1}^n \lambda_i x \lambda_i^*$ for all
  $x \in \mC_\mA(\mB)$ (in particular, this expression is independent of the quasi-basis $\{\lambda_i\}$); 
  \item ${E_0}(\mC_\mA(\mB)) = \mZ(\mB) =\C$; and, 
  \item  $\tau_0:={E_0}_{_{\restriction_{\mC_\mA(\mB)}}}:\mC_\mA(\mB) \to \C$ is a
  faithful tracial state.
    \end{enumerate}

    (2): Let $\{\mu_j\}$ be a two-sided orthonormal quasi-basis for $\tau_0$
    in $\mC_\mA(\mB)$ as in \Cref{tr-2-sided-basis}. We assert that $\{\mu_j\}$ is a two-sided quasi-basis for
    ${E_0}_{\restriction_{\mC}}$ as well. Note that, for any $z
    \in \mC_\mA(\mB)$ and $b \in \mB$, $\{\mu_j\}$ being a right quasi-basis
    for $\tau_0$, we have
    \[
    zb= \sum_j \tau_0(z\mu_j)\mu_j^*b =  \sum_j E_0(z\mu_j)\mu_j^*b = \sum_j E_0(zb\mu_j)\mu_j^*;
    \]
and, $\{\mu_j\}$ being a left quasi-basis for $\tau_0$, we have
 \[
    zb= \sum_j \tau_0(z\mu_j^*)\mu_jb =  \sum_j E_0(z\mu_j^*)\mu_jb = \sum_j E_0(zb\mu_j^*)\mu_j.
    \]
Since $\mC= \overline{\mathrm{span}}\{zb: z \in \mC_\mA(\mB), b\in \mB\}$, it
follows that
\[
\sum_j E_0( w \mu_j)\mu_j^* = w = \sum_j(w \mu_j^*) \mu_j \text{ for
  all } w \in \mC.
\]
Hence, $\{\mu_j\}$ is a two-sided orthonormal quasi-basis for
${E_0}_{\restriction_{\mC}}$.
\end{proof}

Watatani had realized in \cite{Wat} itself that the minimal conditional
expectation is characterized by the tracial property on the
centralizer algebra $\mC_\mA(\mB)$. This observation allowed us in
\cite{BG2} to obtain a sequence of consistent tracial states on the
tower of finite dimensional $C^*$-algebras
\[
\C \cong \mZ(\mB) \subseteq \mC_\mA(\mB)\subseteq \mC_{\mA_1}(\mB) \subseteq
\cdots \subseteq \mC_{\mA_k}(\mB) \subseteq \cdots
\]
which then paved way for a Fourier theory on this tower of centralizer algebras 
- see \cite{BG2, BGS}. 

\subsection{Compatible intermediate $C^*$-subalgebras}

As in  \cite{IW} (also see \cite{BG2, GS}),  for an inclusion $\mB
\subset \mA$ of unital $C^*$-algebras with a finite-index conditional
expectation $E: \mB \to \mA$, let $\mathrm{IMS}(\mB, \mA, E)$ denote the set of
intermediate $C^*$-subalgebras $\mC$ of  $\mB\subset \mA$ with a
compatible conditional expectation $F: \mA \to \mC$ satisfying the
compatibility condition $ E = E_{\restriction_\mC} \circ
F$. 

We shall need the following well known elementary observations.
\begin{remark}\label{unique-compatible-ce} \label{eC-via-e1}
With notations as in the preceding paragraph, let $\mA_1$ denote the
Watatani's $C^*$-basic construction of $\mB \subset \mA$ with respect to
$E$ and $e_1$ denote the corresponding Jones projection. 
  \begin{enumerate}
\item If $\mC\in \mathrm{IMS}(\mB,\mA,E)$ with respect to two compatible
  conditional expectations $F$ and $F'$, then $F = F'$. (See \cite[Page
    3]{IW}.)
  \item If $\mC\in \mathrm{IMS}(\mB,\mA,E)$ with respect to the compatible
    conditional expectation $F: \mA \to \mC$, then $F$ has finite
    index. (See \cite[Remark 2.4]{GS}.)
\item Let $\mC\in \mathrm{IMS}(\mB, \mA, E)$ with respect to the compatible
  conditional expectation $F: \mA \to \mC$; $\mC_1$ denote the Watatani's
  $C^*$-basic construction of $\mC\subset \mA$ with respect to $F$ (with
  Jones projection $e_\mC$), and let $\{\lambda_i\}$ be a quasi-basis
  for $E_{\restriction_\mC}$. Then, $\mC_1 \subset \mA_1$ and $\sum_i
  \lambda_i^* e_1 \lambda_i = e_\mC$. (See \cite[Proposition 2.7]{GS}.)
  \end{enumerate}
\end{remark}

\subsection{Reduced  twisted crossed product}

Recall that a discrete twisted $C^*$-dynamical system is a quadruple $(\mA, G,
\alpha, \sigma)$ consisting of a unital $C^*$-algebra $\mA$, a 
discrete group $G$, a map $\alpha: G \to \mathrm{Aut}(\mA)$ and a map
$\sigma: G\times G \to \mU(\mA)$ satisfying the following identities:
\[
\alpha_g \circ \alpha_h = \sigma(g, h ) \alpha_{gh} \sigma(g,h)^*;\]
\[
\sigma (g, h) \sigma (gh, k) = \alpha_g(\sigma(h,k)) \sigma(g, hk);
\text{and,}\]
\[
\sigma (g, e) = \sigma( e, g) = 1
\]
for all $g, h, k \in G$. Such a $\sigma$ is called a normalized
$\mU(\mA)$-valued $2$-cocycle on $G$, and an $\alpha$ as above is called a
twisted action of $G$ on $\mA$ with respect to the cocycle $\sigma$.

Note that if $\sigma$ is the trivial map, i.e., $\sigma (g,h ) = 1$
for all $g, h \in G$, then $\alpha$ is a homomorphism and $(\mA, G,
\alpha)$ is a usual $C^*$-dynamical system.

We shall work with the following working definition of the reduced twisted crossed product:

For a (discrete) twisted $C^*$-dynamical system $(\mA, G, \alpha,
\sigma)$, there exists a representation $\mA \subset B(\mH)$ and an
injective map $u: G \to \mU(B(\mH))$ such that
\[
u_g u_h =
  \sigma(g, h) u_{gh},
\alpha_g(a) = u_g a
u_g^*
\]
for all $g, h \in G$ and $a \in \mA$; and, the reduced twisted crossed
product $\mA \rtimes^r_{(\alpha, \sigma)} G$ (also denoted by
$C^*_r(\mA, G, \alpha, \sigma)$) can be identified with $C^*(\mA \cup
u(G)) \subset B(\mH)$. For more on reduced twisted crossed product,
we refer the reader to \cite{B1, BO}.

\begin{remark}\label{crossed-product-facts}
With notations as in the preceding paragraph, the following aspects of
the reduced twisted crossed product will be relevant for this article:
\begin{enumerate}
\item  $\Big\{\sum_{\text{finite}} x_g u_g : x_g \in
  \mA\Big\}$  is a unital $*$-subalgebra of $\mA \rtimes^r_{(\alpha, \sigma)} G$ as
  \[
  x u_g y u_h = x \alpha_g(y) u_g u_h = x \alpha_g(y) \sigma(g, h) u_{gh}\]
  for all $x, y \in \mA$, $g,h \in G$;
  and, thus, $\Big\{\sum_{\text{finite}} x_g u_g : x_g \in
  \mA\Big\}$ is dense in $\mA \rtimes^r_{(\alpha, \sigma)} G$.

\item $(u_g)^* = u_{g^{-1}} \sigma(g, g^{-1})^* = \sigma(g, g^{-1})^*
  u_{g^{-1}}$ for all $g \in G$. (\cite[Page 5]{B1})

\item There exists a faithful conditional expectation $E: \mA
  \rtimes^r_{(\alpha, \sigma)} G \to \mA$ such that $E(u_g) = 0$ for all
  $ e \neq g \in G$. (\cite[Page 7]{B1})

\item The canonical conditional expectation is $G$-equivariant,
    i.e., $E(u_g x u_g^*) = \alpha_g (E(x))$ for all $x\in \mA$ and
    $g\in G$. (\cite[Page 8]{B1})
  
  \item If $G$ is finite, then the unital $*$-subalgebra
    $\Big\{\sum_{\text{finite}} x_g u_g : x_g \in \mA\Big\}$ is closed and hence
    \[
    \mA \rtimes^r_{(\alpha, \sigma)}G = \Big\{\sum_{g \in G} x_g u_g : x_g
    \in \mA\Big\}.
    \]
    Moreover, $\{u_g: g \in G\}$ is a quasi-basis for $E$ because if
    $x = \sum_g x_g u_g \in \mA \rtimes^r_{(\alpha, \sigma)} G$, then
    \( E(xu_g^*) = x_g\) for all $g \in G$; so that, $x = \sum_g E(x
    u_g^*) u_g$ for all $x \in \mA \rtimes^r_{(\alpha, \sigma)} G$. In
    particular, $\mathrm{Ind}_W(E) = |G|$.

        \item There is a Galois correspondence between subgroups of $G$
      and intermediate $C^*$-subalgebras of $\mA \rtimes^r_{(\alpha, \sigma)} G$.
      (\cite[Theorem 5.2]{BO})
\end{enumerate}
\end{remark}

\subsection{Some generalities}
\subsubsection{Outer automorphisms and free automorphisms}
Recall that an automorphism $\theta$ of a unital $C^*$-algebra $\mA$ is
said to be free if, for a given $y \in \mA$, $yx = \theta(x)y$ for
every $x \in \mA$ if and only if $y = 0$.

It is easily seen that a free automorphism is outer (i.e., not inner)
and it is well know that an autmorphism of a $II_1$-factor is free if
and only if it is outer. Analogous to this, it can be deduced easily
from \cite{Cho} that the same equivalence holds for any
automorphism of a unital $C^*$-algebra with trivial center. We derive
it here for the sake of convenience.

\begin{lemma}\cite{Cho}\label{outer=free}
  Let $\theta$ be an automorphism of a unital $C^*$-algebra $\mA$ with
  $\mZ(\mA) \cong \C$. Then, $\theta$ is outer if and only if it is free.
  \end{lemma}
\begin{proof}
  We just need to show the necessity. So, let $\theta$ be outer.

  Suppose that there exists an $a \in \mA$ such that $ax = \theta(x) a$ for
  all $x \in \mA$. Then, by \cite[Theorem 1]{Cho}, $aa^*=a^*a \in \mZ(\mA)
  \cong \C$; thus, $aa^*=a^*a=\|a\|^2$. So, if $a \neq 0$,
  then $u:=\|a\|^{-1} a$ is a unitary in $\mA$ and $\theta(x) = uxu^*$ for
  all $ x \in \mA$, which contradicts the outerness of $\theta$. Hence,
  $a=0$, i.e., $\theta$ is free.
    \end{proof}

\subsubsection{Regular inclusions}

Recall that, for an inclusion $\mB \subset \mA$ of unital $C^*$-algebras
with common identity, the normalizer of $\mB$ in $\mA$ is the group of
unitaries given by
\[
\mN_\mA(\mB) = \{u \in \mU(\mA): u\mB u^* =
\mB\};
\]
and, we say that the inclusion $\mB \subset \mA$ is regular if
$\mN_\mA(\mB)$ generates the $C^*$-algebra $\mA$.

\begin{remark}
  It must be mentioned here that Kumjian and Renault use a different
  definition for the normalizer, namely, the set $\{x \in \mA: x\mB
  \subset \mB, \mB x \subset \mB\}$. And, Renault (in \cite{Ren})
  calls an inclusion $\mB \subset \mA$ to be regular if the set $\{x
  \in \mA: x\mB \subset \mB, \mB x \subset \mB\}$ generates $\mA$ as a
  $C^*$-algebra. It has been kindly pointed out to us by Renault (in a
  private communication) that both definitions are equivalent when
  $\mB \subset \mA$ are unital $C^*$-algebras.
\end{remark}

\begin{example}\label{example-regular}
  Consider a  reduced twisted crossed product $\mA \rtimes^r_{(\alpha,
  \sigma)} G$ as in \Cref{crossed-product-facts}. Since $u_g x u_g^* =
\alpha_g(x)$ for all $x \in \mA$, $g \in G$, it follows that $\{u_g : g
\in G\} \subset \mN_{(\mA \rtimes^r_{(\alpha, \sigma)} G)}(\mA)$ and hence
that $\mA \subset \mA \rtimes^r_{(\alpha, \sigma)} G$ is a regular
inclusion. 
\end{example}

\begin{theorem}  (\cite[Theorem 3.2]{B1}, \cite[Theorem 5.1]{BO})\label{bedos-results}
 With notations as in \Cref{crossed-product-facts}, if $\mA$ is simple
 and $\alpha_g$ is outer for every $e \neq g \in G$, then $\mA
 \rtimes^r_{(\alpha, \sigma)} G$ is simple and the inclusion $\mA
 \subset \mA \rtimes^r_{(\alpha, \sigma)} G$ is irreducible.
  \end{theorem}
As a consequence of the main result of this article, we shall see in
\Cref{characterization-irreducible-regular} that every irreducible
regular inclusion of simple unital $C^*$-algebras with a finite-index
conditional expectation arises only in this fashion.

As mentioned in the introduction, the essence of this article lies in
\Cref{regular-characterization}, wherein we establish that every
finite-index regular inclusion of simple unital $C^*$-algebras can be
realized as a cocycle crossed product via an outer action of a finite
group.

  \subsubsection{Finite depth $C^*$-inclusions}
The notion of depth is well established in the theory of
subfactors. Recently, it's analogue in the theory of $C^*$-algebras
has been developed and studied in good detail in \cite{JOPT}. We shall
only need the definition.

Consider an inclusion $\mB \subset \mA$ of unital $C^*$-algebras with
a finite-index conditional expectation $E: \mA \to \mB$. Then,
consider the Watatani's $C^*$-basic construction $\mB \subset \mA
\subset \mA_1$ with respect to the conditional expectation $E$.  We
know that the dual conditional expectation $E_1: \mA_1 \to \mA$ also
has finite index (\Cref{dual-ce}). Thus, one can iterate the
$C^*$-basic construction to obtain a tower of unital $C^*$-algebras
\[
\mB \subset \mA \subset \mA_1 \subset \cdots \subset \mA_k \subset \cdots,
\]
where, for each $k \geq 0$, $\mA_{k+1} = C^*(\mA_k \cup \{e_{k+1}\})$
denotes the $C^*$-basic construction of the inclusion $\mA_{k-1}\subset
\mA_{k}$ with respect to the finite-index conditional expectation $E_k:
\mA_k \to \mA_{k-1}$, with $\mA_0:= \mA$ and $\mA_{-1} := \mB$.

The inclusion $\mB \subset \mA$ is said to have finite depth if $(\mB'\cap
\mA_{k}) = (\mB'\cap \mA_{k-1}) e_{k} (\mB'\cap \mA_{k-1})$ for some $k \geq 1$. The
least such $k$ is called the depth of $\mB \subset \mA$. We shall show in \Cref{regular-depth-2} that every regular inclusion of simple unital $C^*$-algebras  with a finite-index conditional expectation has depth at most $2$.
We refer the reader to \cite{JOPT} for more on finite depth
inclusions of $C^*$-algebras.

\section{Structure of  regular inclusions of simple $C^*$-algebras}

\subsection{Generalized Weyl group of an inclusion of $C^*$-algebras}\( \)

Note that, if $\mB \subset \mA$ is an inclusion of unital $C^*$-algebras
with common unit, then $\mU(\mB)$ and $\mU(\mC_\mA(\mB))$ are both normal
subgroups of $\mN_\mA(\mB)$ and, hence, $\mU(\mB) \mU(\mC_\mA(\mB)) (=
\mU(\mC_\mA(\mB))\mU(\mB)) $ is also a normal subgroup of
$\mN_\mA(\mB)$.

Analogous to the notions of the Weyl group and the
generalized Weyl group of an inclusion of von Neumann algebras (see
\cite[Definition 2.11]{BG1}), we make the following definitions:

\begin{definition}\label{weyl-defn}
  Let $\mB \subset \mA$ be an inclusion of unital $C^*$-algebras with
  common unit. Then,
  \begin{enumerate}
\item the  Weyl
  group of the inclusion $\mB \subset \mA$ is defined as the quotient group \( \frac{\mN_\mA(\mB)}{\mU(\mB)
   } \) and will be denoted by $W_0(\mB
  \subset \mA)$; and,
  \item the   generalized Weyl
  group of the inclusion $\mB \subset \mA$ is defined as the quotient group \( \frac{\mN_\mA(\mB)}{\mU(\mB)
    \mU(\mC_\mA(\mB))} \)  and will be denoted by $W(\mB
  \subset \mA)$.
\end{enumerate}
  \end{definition}

Clearly, there exists a canonical surjective
homomorphism from the Weyl group onto the generalized Weyl group. And,
if $\mB \subset \mA$ is an irreducible inclusion, i.e., $\mB^{\prime}\cap
\mA=\C$, then $W(\mB \subset \mA) = W_0(\mB \subset \mA)$. In this article, we
shall focus only on the generalized Weyl group.

The following elementary observation will be used ahead. 
 \begin{lemma}\label{orthogonal-system}
  Let $ \mB \subset \mA$ be an inclusion of unital $C^*$-algebras and $w$
  be a unitary in $\mN_\mA(\mB) \setminus \mU(\mB) \mU(\mC_\mA(\mB))$. Then,
  $\mathrm{Ad}_w$ is an outer automorphism of $\mB$.

  Moreover, if $\mZ(\mB) \cong \C$, then $\mathrm{Ad}_w$ is a free
  automorhpism of $\mB$ and, for any conditional expectation $E: \mA\to
  \mB$, $E(w) = 0$.  In particular, for any two elements $u, v \in
  \mN_\mA(\mB)$, $E(vu^*) = 0 = E(v^*u)$ if $[u] \neq [v]$ in  $W(\mB \subset \mA)$.
   \end{lemma}

 \begin{proof}
We have $w\mB w^* = \mB$. Let $\varphi = \text{Ad}_w$ and suppose, on
contrary, that $\varphi$ is not an outer automorphism of $\mB$. Then,
there exists a $v\in \mU(\mB)$ such that $wxw^*=vxv^*$ for all $x \in
\mB$. This implies that $v^*w \in \mU(\mC_\mA(\mB))$; so that $w \in
\mU(\mB)\mU(\mC_\mA(\mB))$, which is a contradiction. Thus,
$\text{Ad}_w$ must be an outer automorphism of $\mB$.

If $\mZ(\mB) \cong \C$, then, by \Cref{outer=free}, it follows that
$\varphi$ is free as well. Further, let $E: \mA \to \mB$ be a conditional
expectation. Since $\varphi(x) w = w x $ for all $ x \in \mB$, it
follows that
  \[
  \varphi(x) E(w) = E(w) x \text{  for all $x \in \mB$}.
  \]
  So, by freeness of $\varphi$, we must have  $E(w) = 0$.
 \end{proof}

 A priori, it is not clear whether the generalized Weyl group of an
 inclusion $\mB \subset \mA$ is finite or not. However, if $\mB$ has trivial
 center and there exists a finite-index conditional expectation from
 $\mA$ onto $\mB$, we can show that $W(\mA \subset \mB)$ is finite and provide
 a bound for its cardinality.

\begin{proposition}\label{weyl-group} 
Let $\mB \subset \mA $ be an inclusion of unital $C^*$-algebras with
$\mZ(\mB) \cong \C$. If there exists a finite-index conditional
expectation $E: \mA \to \mB$, then $W(\mB \subset \mA)$ is finite and
\[
|W(\mB \subset \mA)| \leq \mathrm{dim}(\mB'\cap \mA_1),
\]
where $\mA_1
$ is the $C^*$-basic construction of $\mB \subset \mA$ with respect to
the conditional expectation $E$.
\end{proposition}

\begin{proof}
Let $G:=W(\mB \subset \mA)$ and $\{u_g: g \in G\}$ denote a set of
coset representatives of $G$ in $\mN_\mA(\mB)$. Also, let $e_1$ denote
the Jones projection corresponding to $E$.

 We first assert that $\{u_g e_1 u_g^* : g \in G\}$ is a collection of
 mutually orthogonal projections in the algebra
 $\mB'\cap \mA_1$.

Note that, for each $g \in G$, since $e_1 \in \mB'$ and $u_g^* \mB u_g =
\mB$, we have
\[
(u_g e_1 u_g^*) x = u_g e_1 (u_g^* x u_g) u_g^* = u_g (u_g^* x u_g)
e_1 u_g^* = x (u_g e_1 u_g^*)
\]
for all $ x\in \mB$. Hence, $u_g e_1 u_g^* \in \mB'\cap \mA_1$ for all $g
\in G$. Further, by \Cref{orthogonal-system}, we observe that
\[
(u_g e_1 u_g^*) (u_h e_1 u_h^*) = u_g E(u_g^* u_h) e_1 u_h^* =
\delta_{g, h}u_g e_1 u_g^*
\]
for all $g, h \in G$. This proves our assertion (for which we did not
require $E$ to have finite index).

Next, let $\tilde{E}: \mA_1 \to \mA$ denote the dual conditional
expectation of $E$.  Then, $\tilde{E}$ and (hence) $\tilde{E} \circ E:
\mA_1 \to \mB$ are finite-index conditional expectations - see
\Cref{fi-ce-facts}.  Thus, it follows from \cite[Proposition
  2.7.3]{Wat} (also see \cite[Proposition 2.16]{BG2}) that $\mB'\cap
\mA_1$ is finite dimensional. Hence, $G$ must be finite and $|G| \leq
\mathrm{dim}(\mB'\cap \mA_1)$.
\end{proof}

\subsection{Two-sided and unitary bases for regular inclusions}

The following is an obvious adaptation of \cite[Lemma
  3.5]{BG1}.  We  skip the proof.

\begin{lemma}\label{outer=free-C}\cite{BG1}
Let $\mB \subset \mA $ be an inclusion of simple unital $C^*$-algebras
with a finite-index conditional expectation $E: \mA \to \mB$ and
let $\mC$ denote the intermediate $C^*$-subalgebra generated by $\mB$ and
its centralizer $\mC_\mA(\mB)$.  If $\theta$ is an automorphism of $\mC$
whose restriction to $\mB$ is an outer automorphism of $\mB$, then
$\theta$ is a free automorphism of $\mC$.
\end{lemma}

\begin{corollary}\label{orthogonal-system-C}
  Let the notations be as in \Cref{outer=free-C}. Then, $\mN_\mA(\mB)
  \subseteq \mN_\mA(\mC_\mA(\mB)) \cap \mN_\mA(\mC)$ and, for each $w
  \in \mN_\mA(\mB) \setminus \mU(\mB) \mU(\mC_\mA(\mB))$, $w \mC w^*
  = \mC$, $\mathrm{Ad}_w$ is a free automorphism of $\mC$ and
  $E(w) = 0$.

In particular, for any two elements $u, v \in \mN_\mA(\mB)$, $E(vu^*)
= 0 = E(v^*u)$ if $[u] \neq [v]$ in the generalized Weyl group $W(\mB
\subset \mA)$.
  \end{corollary}
\begin{proof}
  That $\mN_\mA(\mB) \subseteq \mN_\mA(\mC_\mA(\mB)) \cap \mN_\mA(\mC)$ follows on
  the lines of the proof of \cite[Lemma 3.2]{BG1}. Further, it readily
  follows that $w \mC w^* = \mC$. Then, because of the preceding lemma,
  the rest follow on the lines on \Cref{orthogonal-system}.
  \end{proof}

\begin{lemma}\label{IMS-F}
  Let the notations be as in \Cref{outer=free-C} and $E_0: \mA \to
  \mB$ denote the minimal conditional expectation as in \Cref{minimal}. Then, $\mC\in
  \mathrm{IMS}(\mB, \mA, E_0)$.
 \end{lemma}
\begin{proof}
    Since $\mA$ and $\mB$ are simple and $E_0: \mA \to \mB$ has finite index,
    it follows from \cite[Proposition 6.1]{Iz} that there exists a
    (finite-index) conditional expectation $F: \mA \to \mC$. We assert that
    $F$ is compatible with respect to $E_0$.

  Note that ${(E_0)}_{\restriction_\mC}:\mC\to \mB$ has finite index, by
  \Cref{E0-restriction-to-C-two-sided}. Thus, in view of \cite[Lemma
    2.12.2]{Wat}, it is enough to show that the restrictions of $E_0$
  and $(E_0)_{\restriction_\mC} \circ F$ to $\mC_\mA(\mB)$ are same.
  Clearly, $\mC_\mA(\mB) = \mB'\cap \mC$ because $\mC$ contains
  $\mC_\mA(\mB)$; and, for any $z \in \mB'\cap \mC$, we have
  \[
  ({(E_0)}_{\restriction_\mC} \circ F)(z) = {(E_0)}_{\restriction_\mC}(z) =
  E_0(z).
  \]
  Hence, $F$ is a compatible conditional expectation and $\mC\in
  \mathrm{IMS}(\mB, \mA , E_0)$.
\end{proof}

We shall now focus only on regular inclusions of simple unital
$C^*$-algebras. Thus, from here on, $\mB \subset \mA$ denotes a fixed
regular inclusion of simple unital $C^*$-algebras with a finite-index
conditional expectation from $\mA$ onto $\mB$.

Let $E_0: \mA \to \mB$ denote the (unique) minimal conditional expectation
as in \Cref{minimal} and $G$ denote the generalized Weyl group of the
inclusion $\mB \subset \mA$ with a fixed set of left coset
representatives $\{u_g : g \in G\}$ in $\mN_\mA(\mB)$, with $u_e =
1$. Further, let  $\mC:=C^*(\mB \cup (\mC_\mA(\mB)))$; $F: \mA \to \mC$
be the compatible conditional expectation as in \Cref{IMS-F} and, let
$\mA_1$ denote the Watatani's $C^*$-basic construction of $\mB \subset \mA$
with respect to the minimal conditional expectation $E_0$ (and Jones
projection $e_1 \in \mA_1$).
\begin{proposition} \label{cosetrepresentative}
$\{u_g: g \in G\}$ is a two-sided orthonormal quasi-basis for $F$.

  In particular,  $\mathrm{Ind}_W(F) = |G|$.
 \end{proposition}

\begin{proof}
  We first assert that $\sum_{ g\in G} u_g\mC= \mA$.

  Let $\mL:= \sum_{ g\in G} u_g\mC$.  Note that, for any two $g, h$ in
  $G$, $(u_g\mC)( u_h \mC) = u_k\mC$ for some $k \in G$. Also, for any
  $g \in G$, we have $(u_g\mC)^*= \mC u_g^* = u_g^* u_g \mC u_g^* =
  u_g^* \mC$ because $u_g \in \mN_\mA(\mB) \subseteq \mN_\mA(\mC)$ (by
  \Cref{orthogonal-system-C}); so that, $(u_g\mC)^* = u_k \mC$ for
  some $k \in G$. Hence, $\mL$ is a unital $*$-subalgebra of
  $\mA$. Further, since $\mB \subset \mA$ is regular, it follows that $\mL$
  is dense in $\mA$ (because $\mN_\mA(\mB) = \cup_{g \in G} u_g \mC$).
  \color{black} So, it just remains to show that $\mL$ is closed.

  Let $a \in \overline{\mL}$. Then, $\sum_{g \in G} u_g c_g^{(n)} \to a$
  for some sequence $\{ \sum_g u_g c_g^{(n)}\} \subset L$. Thus,
  \[
  F\Big(u_h^* \sum_{g \in G} u_g c_g^{(n)}\Big) \to F(u_h^* a)
  \] 
  for all $h
  \in G$. Note that, by \Cref{orthogonal-system-C}, we have
  $F(u_s^*u_t) = \delta_{s, t}$ for all $s, t \in G$; so,
\[
F\Big(u_h^* \sum_{g \in G} u_g c_g^{(n)}\Big) = \sum_g F(u_h^* u_g )
c_g^{(n)} = c_h^{(n)}
\]
for all $h \in G$, $n \in \N$. Thus, $\sum_g
u_g c_g^{(n)} \to \sum_g u_g F(u_g^* a) \in \mL$; so, $a \in \mL$ and $\mL$
is closed. This proves our assertion.

Now,  every $x\in \mA$ can be written as $x = \sum_g u_g
c_g$, $c_g \in \mC$. Thus, $F(u_h^*x) = \sum_g F(u_h^* u_g) c_g = c_h$
for all $h \in G$; so that $x = \sum_g u_gF(u_g^*x)$ for all $x \in
\mA$. Also, by \Cref{orthogonal-system-C} again, we have $F(u_g^* u_h) =
\delta_{g,h}$ for all $g, h \in G$. Hence, $\{u_g : g \in G\}$ is an
orthonormal right quasi-basis for $F$.

Again, since $\{u_g \}\subseteq\mN_\mA(\mB) \subseteq \mN_\mA(\mC)$, we have
$u_g \mC= \mC u_g$ for all $g \in G$. So, $\sum_g \mC u_g = \mA$. And, as
above, it is easily seen  that $x = \sum_g F(xu_g^*)u_g$ for all $x \in
\mA$. Hence, $\{u_g\}$ is an orthonormal left quasi-basis for $F$ as
well.

Thus, $\{u_g: g \in G\}$ is a two-sided orthonormal quasi-basis for
$F$ consisting of unitaries in $\mN_\mA(\mB)$. Finally, we have \(
\mathrm{Ind}_W(F) = \sum_{g\in G} u_g u_g^* = |G|.  \)
\end{proof} 

\begin{theorem}\label{theorem-regular}\label{weylgroup}
  With running notations, the following hold:
  \begin{enumerate}
    \item $E_0: \mA \to \mB$ admits a two-sided orthonormal quasi-basis;
  \item $\tau_0:= {(E_0)}_{\restriction_{\mC_{\mA}(\mB)}} :\mC_{\mA}(\mB) \to \C$ is
    the (unique) Markov trace for the (connected) inclusion $\C
    \subseteq \mC_{\mA}(\mB)$ with modulus $\mathrm{dim}(\mC_{\mA}(\mB))$; and,
\item
  \(
    [\mA:\mB]_0 =  |W(\mB \subset \mA)|\, \mathrm{dim}(\mC_\mA(\mB)).
    \)
  \end{enumerate}
  In particular, $[\mA:\mB]_0$ is an integer, $|W(\mB \subset \mA)| \leq
  [\mA:\mB]_0$ and if, in addition, $\mB \subset \mA$ is irreducible, then
    \[
      [\mA:\mB]_0 =  |W_0(\mB \subset \mA)|.
      \]
\end{theorem}
\begin{proof}
(1): From \Cref{cosetrepresentative}, $\{u_g : g \in G\}$ is a
 a two-sided orthonormal  quasi-basis for $F$ contained in
  $\mathcal{N}_\mA(\mB)$. And, from \Cref{E0-restriction-to-C-two-sided},
  there exists a two-sided orthonormal quasi-basis for
  ${E_0}_{\restriction_\mC}$ contained in $\mC_\mA(\mB)$, say, $\{\lambda_i:
  1 \leq i \leq n\}$. We assert that $\{u_g \lambda_i: g \in G, 1\leq
  i \leq n\}$ is a two-sided orthonormal quasi-basis for $E_0$.

  Since $E_0 = {E_0}_{\restriction_\mC} \circ F$, it follows easily that
  $\{ u_g \lambda_i: 1 \leq i \leq n, g \in G\}$ is a right
  orthonormal quasi-basis for $E_0$ - see \Cref{fi-ce-facts}. So, we
  just need to show that it is a left orthonormal quasi-basis as well,
  equivalently, $\{\lambda_i^*  u_g^*\}$ is a right orthonormal
  quasi-basis for $E_0$. Clearly,
  \[
  E_0((\lambda_i^* u_g^*)^* \lambda_j^*
  u_h^*) = E_0(u_g F(\lambda_i \lambda_j^*) u_h^*) = \delta_{i,j}
  \delta_{g, h}.
  \]
 So, in view of \Cref{basis-characterization}, it suffices
  to show that $\sum_{g, i} \lambda_i^* u_g^* e_1 u_g \lambda_i = 1$.

For each $g \in G$, it readily follows that $\{u_g \lambda_i u_g^*:
i\}$ is also a two-sided quasi-basis for ${E_0}_{\restriction_\mC}$ -
see \cite[Lemma 3.8]{BG1}. Let $e_\mC$ denote the Jones projection for
the inclusion $\mC\subset \mA$ with respect to the finite-index
conditional expectation $F$. Then, $e_\mC\in \mA_1$ and $\sum_i (u_g\lambda_iu_g^*)^*
e_1 \lambda_i (u_g \lambda_i u_g^*)= e_\mC$ (by \Cref{eC-via-e1}). Thus,
\[
\sum_{g, i } \lambda_i^* u_g^* e_1 u_g \lambda_i = \sum_{g, i} u_g^*
(u_g\lambda_i u_g^*)^* e_1 (u_g\lambda_i u_g^*) u_g = \sum_g u_g^* e_\mC
u_g = 1,
  \]
  where last equality follows from \Cref{basis-characterization}. This
  proves (1).\smallskip

  (2) and (3): Since $E_0 = {E_0}_{\restriction_\mC} \circ F$, it
  follows that $\mathrm{Ind}_W(E_0) = \mathrm{Ind}_W(F)
  \mathrm{Ind}_W({E_0}_{\restriction_\mC})$ - see
  \Cref{fi-ce-facts}. Hence, from \Cref{cosetrepresentative} and
  \Cref{tr-2-sided-basis}, we obtain
    \[
|G| \mathrm{dim}(\mC_\mA(\mB)) = [\mA : \mB]_0= \sum{g, i} \lambda_i^* u_g^* u_g
\lambda_i = \mathrm{Ind}_W(\tau_0).
\]
In particular, $\tau_0$ has scalar Watatani index.  Thus, in view of
\Cref{tr-2-sided-basis}, it follows that $\tau_0$ is the Markov trace
for the inclusion $\C \subseteq \mC_\mA(\mB)$ with modulus
$\mathrm{dim}(\mC_\mA(\mB))$.
\end{proof}

\begin{corollary}
 Let $\mB$ be a simple unital $C^*$-algebra and suppose a finite group
 $K$ admits a $\mathcal{U}(\mB)$-valued outer cocycle  action  $(\alpha, \sigma)$ on $\mB$. Then, the Weyl group of the inclusion $\mB
 \subset \mB\rtimes^r_{(\alpha, \sigma)} K$ is isomorphic to $K$.
  \end{corollary}

\begin{proof}
 We simply write $\mB \rtimes K$ for the reduced twisted crossed product
 $\mB \rtimes^r_{(\alpha, \sigma)} K$.

Note that, by the universality of the reduced twisted crossed product, we can assume that there exists
a Hilbert space $\mH$ such that $\mB \subseteq B(\mH)$ and there is a
map $w: K \to \mU(B(\mH))$ such that $w_s \notin \mB$ (for $s \neq
e$),
\[
w_s w_t = \sigma(s, t) w_{st} , w_e= 1, (w_s)^*=
\sigma(s^{-1}, s)^* w_{s^{-1}}, \alpha_s(x) = w_s x w_s^*
\]
for all $s, t \in K$ and $x \in \mB$; and, $\mB \rtimes K = C^*(\mB, w(K))
\subset B(\mH)$.

Since the (twisted) action is outer, $\mB \rtimes K$ is simple and
$\mB'\cap (\mB \rtimes K) = \C$, i.e., $\mB \subset \mB \rtimes K$ is
irreducible (see \Cref{bedos-results}). Also, $\mB \subset \mB \rtimes
K$ is regular (see \Cref{example-regular}); so, $|W_0(\mB \subset \mB
\rtimes K)| = [\mB \rtimes K: \mB]_0$,  by \Cref{theorem-regular}.

Further,
since $\mB \subset \mB \rtimes K$ is irreducible, the canonical
conditional expectation $E: \mB \rtimes K \to \mB$ (as in
\Cref{crossed-product-facts}(3)) is unique (by \cite[Corollary 1.4.3]{Wat})
and hence minimal, which then implies that $[\mB \rtimes K: \mB]_0 = \Ind_W(E)= |K|$, by
\Cref{crossed-product-facts}(5). Thus, $|W_0(\mB \subset \mB
\rtimes K)| = |K|$. 

Finally, $\{w_s : s \in K\} \subset \mN_{\mB \rtimes K}(\mB)$
and the map $K \ni s \mapsto [w_s] \in W_0(\mB \subset \mB \rtimes K)$
is an injective group homomorphism.  Hence, $W_0(\mB \subset \mB
\rtimes K) \cong K.$
  \end{proof}
It will be interesting to answer the following natural question.\smallskip

\noindent {\bf Question:} Suppose a (countable) discrete group $G$
admits a cocycle action $(\alpha, \sigma)$ on a unital $C^*$-algebra
$\mB$. Is the generalised Weyl group of the inclusion $\mB \subset \mB
\rtimes^r_{(\alpha, \sigma)} G$ isomorphic to $G$?
\smallskip \color{black}

The first part of the following observation now follows from \cite{CKP}, which 
is based on  a beautiful application of the so-called ``circulant matrices''
from Quantum Information Theory.
\begin{corollary}\label{E0-restriction-to-C-unitary}
  With running notations, the following hold:
  \begin{enumerate}   
  \item There exists a unitary orthonormal quasi-basis for $\tau_0$.
\item  The conditional   expectation ${E_0}$ admits a
    unitary orthonormal quasi-basis.
  \end{enumerate}
\end{corollary}

\begin{proof}
(1): Let $(n_1, \ldots, n_k)$ be the dimension vector of $\mC_A(\mB)$;
  so, $\mC_\mA(\mB) \cong \oplus_{i=1}^k M_{n_i}(\C)$. By
  \Cref{theorem-regular}, $\tau_0$ is the Markov trace for the
  inclusion $\C \subseteq \mC_\mA(\mB)$ with modulus $d =
  \mathrm{dim}(\mC_\mA(\mB))$. So, the trace vector of $\tau_0$ is given
  by $\bar{t}= (\frac{n_1}{d}, \ldots, \frac{n_k}{d})^t$.

Note that $ \mC_\mA(\mB)$ is unitally $*$-isomorphic to $ P:=\oplus_{i=1}^k
\left(I_{n_i}\otimes M_{n_i}(\C)\right)$ and $P$ is a unital
$C^*$-subalgebra of $ M_d(\C)$. Further, if $\tau$ denotes the (unique) tracial
state on $M_d$, then the trace vector of $\tau_{\restriction_P}$ is
given by $(\frac{n_1}{d}, \ldots, \frac{n_k}{d})^t$.  Thus,
$\tau_{\restriction_P}$ corresponds to $\tau_0$ via the
$*$-isomorphism between $P$ and $\mC_\mA(\mB)$. By \cite[Theorem
  2.2]{CKP}, there exists a unitary orthonormal quasi-basis for
$\tau_{\restriction_P}$. Hence, there exists a unitary orthonormal
quasi-basis for $\tau_0$.\smallskip

    (2): Let $\{w_i: 1 \leq i \leq n\}\subset \mC_\mA(\mB)$ be a unitary
orthonormal quasi-basis for $\tau_0$. Then it follows on the lines of
\Cref{E0-restriction-to-C-two-sided}(2) that $\{w_i\}$ is a unitary
orthonormal quasi-basis for ${E_0}{_{\restriction_C}}$ as well. From
\Cref{cosetrepresentative}, we know that $\{u_g : g \in G\}$ is a
unitary orthonormal quasi-basis for $F: \mA \to \mC$. Hence, $\{u_g w_i
: g \in G,1 \leq i \leq n\}$ is a unitary orthonormal quasi-basis for
$E_0$.
  \end{proof}

With all requirements at our disposal, imitating the proof of \cite{BG3} we obtain following:
\begin{theorem}\label{regular-depth-2}
Let $\mB \subset \mA$ be a regular inclusion of simple unital
$C^*$-algebras with a finite-index conditional expectation from $\mA$
onto $\mB$. Then, the inclusion $\mB \subset \mA$ has finite depth and
the depth is at most 2.
  \end{theorem}

\subsection{Characterization of regular inclusions of simple $C^*$-algebras}

We are now all set to prove the main result of this article.
\begin{theorem}\label{regular-characterization}
Let $\mB \subset
  \mA$ be an inclusion of simple unital $C^*$-algebras with a
  finite-index conditional expectation from $A$ onto $\mB$. Then, the
  inclusion $\mB \subset \mA$ is regular if and only if there exists a
  finite group $G$ that admits a cocycle action $(\alpha,\sigma)$ on
  the intermediate $C^*$-subalgebra $\mC:=C^*(\mB \cup \mC_\mA(\mB))$ such
  that \begin{enumerate}
  \item $\mB$ is invariant under $\alpha$;
  \item for each $e \neq g \in G$,  $\alpha_g$ is an outer automorphism of $\mB$;  and,
    \item $(\mB \subset \mA) \cong
      ( \mB \subset \mC\rtimes_{(\alpha, \sigma)}^r G)$.
\end{enumerate}
  \end{theorem}

\begin{proof}
Suppose that a finite group $G$ admits a cocycle action $(\alpha,
\sigma)$ on $\mC$ as in the statement. We can consider a representation
$\mC\subseteq B(\mH)$ such that $\alpha$ is implemented by a map $w: G
\to \mathrm{Aut}(B(\mH))$, i.e., $ \alpha_g = \mathrm{Ad}(w_g)$ for all
$ g \in G$. So, $\mC\rtimes^r_{(\alpha, \sigma)} G = C^*( \mC\cup w(G))$.
The regularity of $ \mB \subset
{\mC\rtimes^r_{(\alpha, \sigma)} G}$ is clear as $\{w_g : g \in G\}\cup
\mathcal{U}(\mC) \subseteq \mN_{\mC\rtimes^r_{(\alpha, \sigma)} G}(\mB)$.
\smallskip

Converserly, suppose that the inclusion $\mB \subset \mA$ is
regular. Consider its generalized Weyl group $G$. Then, $G$ is finite
by \Cref{weyl-group}. We assert that $G$ admits a desired cocycle
action on $\mC$.

Let $\{u_g:g=[u_g]\in G\}$ denote a fixed set of (left) coset
representatives of $G$ in $\mN_{\mA}(\mB)$.  Since $u_g \mC u_g^* =
\mC$ (see \Cref{orthogonal-system-C}), $\alpha_g:= \mathrm{Ad}_{u_g}$
is an automorphism of $\mC$ for every $g \in G$. Moreover, from
\Cref{orthogonal-system-C} again, it follows that $\alpha_g: \mC\to
\mC$ is (free and hence) outer for every $g \neq e$. We assert that
the map $\alpha: G \to \mathrm{Aut}(\mC)$, $g \mapsto \alpha_g$, is in
fact a cocycle action with respect to a $\mU(\mC)$-valued cocycle
$\sigma$, which we describe now.

Note that $[u_gu_h]=[u_{gh}]$ for all $g, h \in G$. So, there
exists a function
\[
\sigma: G \times G \to \mU(\mB) \mU(\mC_\mA(\mB)) \subset
\mU(\mC)
\]
satisfying $u_gu_h=\sigma(g,h)u_{gh}$ for all $g, h \in G$. We assert
that $(\alpha, \sigma)$ is a cocycle action of $G$ on $\mC$.  First,
observe that,
\[
  \alpha_g\alpha_h(x) = \alpha_g(u_hxu^*_h)= u_gu_hxu^*_hu^*_g =
  \text{Ad}\big(\mu(g,h)\big)u_{gh}x u^*_{gh} =
  \text{Ad}\big(\mu(g,h)\big) \alpha_{gh}(x)
\]
 for all $g, h \in G$, $x \in \mC$. Thus, $\alpha_{g} \alpha_h =
 \text{Ad}\big(\sigma(g,h)\big) \alpha_{gh} $ for all $g, h \in G$.

 Now applying the relation $u_gu_h=\sigma(g,h)u_{gh}$ twice, we see that 
 \[
 (u_gu_h)u_k=\sigma(g,h)\sigma(gh,k)u_{(gh)k}; 
 \]
and, on the other hand,
 \[
 u_g(u_hu_k)=u_g \big(\sigma(h,k)u_{hk} \big) =
 \text{Ad}(u_g)\big(\sigma(h,k)\big)\sigma(g,hk)u_{g(hk)}
 \]
  for all $g, h , k \in G$.  Thus,
 \[
 \sigma(g,h)\sigma(gh,k)=\alpha_g\big(\sigma(h,k)\big)\sigma(g,hk)
 \]
 for all $g, h , k \in G$.  And, clearly,
 $\sigma(g,1)=\sigma(1,g)=1$. This proves our assertion that $(\alpha,
 \sigma) $ is a cocycle action of $G$ on $\mC$.

 Further, note that $u_g \mB u_g^* = \mB$ for all $g \in G$ and, for each
 $e \neq g \in G$, $\mathrm{Ad}_{u_g} : \mB \to \mB$ is outer by
 \Cref{orthogonal-system}.
 
Finally, we show that there exists a $*$-isomorphism $\varphi$ from
$\mA$ onto $\mC\rtimes^r_{(\alpha,\sigma)} G$ such that
$\varphi|_{\mC}=\mathrm{id}_\mC$. Let $x\in \mA$. By
\Cref{cosetrepresentative}, we see that $x=\sum_g
F(xu^*_g)u_g$. Define $\varphi: \mA \to \mC\rtimes^r_{(\alpha, \sigma)}
G$ by $\varphi(x) = \sum_g E(xg^{-1}) g$, where $E: \mC
\rtimes^r_{(\alpha, \sigma)} G \to \mC$ is the canonical finite-index
conditional expectation (as in \Cref{crossed-product-facts}).
Clearly, $\varphi|_{\mC}=\mathrm{id}_\mC$ and it is easy to check that
$\varphi$ is a unital $*$-homomorphism. Since the inclusion $\mB \subset
\mA$ is regular, $\varphi$ is surjective as well.  We omit the necessary
details. Since $\mA$ is simple, $\varphi$ is injective and we are done.
   \end{proof}

\begin{corollary}\label{characterization-irreducible-regular}
  Let $\mB \subset \mA$ be a regular 
irreducible  inclusion of simple unital $C^*$-algebras with a finite-index conditional expectation from $\mA$
onto $\mB$.  Then, its Weyl group $G$ admits an outer cocycle action $(\alpha, \sigma)$ on $\mB$ such that $(\mB
\subset \mA) \cong (\mB \subset \mB \rtimes^r_{(\alpha, \sigma)} G)$.
\end{corollary}

\begin{remark}
  Two results from subfactor theory which are very relevant to the preceding
  theorem need to be mentioned here:
\begin{enumerate}
\item Choda (in \cite[Theorem 4]{Cho2}), based on one of his earlier
  techniques in \cite[Theorem 7]{Cho1}, had proved that, for any factor
  $M$ with separable predual, for every irreducible regular subfactor $N
  \subset M$ with a faithful conditional expectation from
  $M$ onto $N$, there exists a countable discrete group $G$ which
  admits an outer cocycle action $(\sigma, \omega)$ on $N$ such that $M \cong
  N \rtimes_{(\sigma, \omega)} G$.

\item Later, employing the same technique of Choda (from \cite[Theorem
  7]{Cho1}), Cameron (in \cite[Theorem 4.6]{Cam}) showed that given
  any regular inclusion $N \subset M$ of $II_1$-factors, there exists
  a countable discrete group $G$ which admits a cocycle action
  $(\sigma, \omega)$ on $Q$, the von Neumann algebra generated by $N$
  and $N'\cap M$, such that $M \cong Q \rtimes_{(\sigma, \omega)}
  G$. However, Cameron does not mention whether the action is outer or
  not.
\end{enumerate}
\end{remark}

       \end{document}